\documentclass[notheoremnums]{dpreprint}

\newtheorem{theorem}{Theorem}[section]
\newtheorem*{theorem*}{Theorem}
\newtheorem{corollary}[theorem]{Corollary}
\newtheorem*{corollary*}{Corollary}
\newtheorem{proposition}[theorem]{Proposition}
\newtheorem{lemma}[theorem]{Lemma}
\theoremstyle{definition}
\newtheorem{definition}[theorem]{Definition}

\def\includeappendix{true}

\begin{document}

\dtitle[Stationary Actions of Lattices]{Stabilizers of Stationary Actions of Lattices in Semisimple Groups}
\dauthor[Darren Creutz]{Darren Creutz}{creutz@usna.edu}{U.S.~Naval~Academy}{Supported in part by the Naval Academy Research Council}
\datewritten{30 November 2020}

\dabstract{%
Every stationary action of a strongly irreducible lattice or commensurator of such a lattice in a general semisimple group, with at least one higher-rank connected factor, either has finite stabilizers almost surely or finite index stabilizers almost surely.  Consequently, every minimal action of such a lattice on an infinite compact metric space is topologically free.
}

\makepreprint
\Large

\section*{Introduction}

Nevo-Stuck-Zimmer \cite{SZ}, \cite{nevozimmerpreprint}, \cite{NZ} proved that every ergodic measure-preserving action of a lattice in a connected semisimple Lie group, all of whose simple factors are higher-rank, either has finite stabilizers almost everywhere or finite index stabilizers almost everywhere.

The author and Peterson \cite{CP17} extended this to lattices in general semisimple groups, and the author \cite{Cr17} relaxed the higher-rank requirement to a single factor.

However, for nonamenable groups such as lattices, measure-preserving actions is not the correct context as actions on compact metric spaces will not, in general, admit invariant measures.

The natural setting for the study of actions of lattices is that of stationary actions (as stationary measures do always exist).  There is a natural measure on a lattice (related to the Poisson boundary of the ambient group) and ww say an action of a lattice is stationary when it is stationary with respect to that measure (see Definition \ref{D:statactlattice}).  Corollary \ref{C:generallattices} states:

\begin{theorem*}
Let $\Gamma < G$ be a strongly irreducible lattice in a semisimple group with finite center and no compact factors, at least one simple factor of higher-rank.  (Strong irreducibility meaning the projection of the lattice is dense in every proper subproduct of $G$).

Every ergodic stationary action of $\Gamma$ has finite stabilizers almost everywhere or finite index stabilizers almost everywhere.
\end{theorem*}

We also obtain the same dichotomy for dense commensurators of lattices (Corollary \ref{C:commensurators}).

As every action on a compact metric space admits a stationary measure we obtain: every action of $\Gamma$ on a minimal compact metric space is either topologically free or the space is finite.  Consequently, every ergodic uniformly recurrent subgroup is finite (see \cite{GW14} for the notion of uniformly recurrent subgroups), answering the generalized form of the question posed in \cite{GW14} as Problem 5.4.

Omitting the requirement of a higher-rank factor but requiring at least two factors, every stationary action either has finite stabilizers or is orbit equivalent to an action of $\mathbb{Z}$.

\subsection*{Stationary Intermediate Factor Theorems and Induced Actions}

We establish an intermediate factor theorem along the lines of those of Stuck-Zimmer \cite{SZ} for stationary actions (Theorem \ref{T:factorprod}): every intermediate factor between the stationary space and the stationary join of the space with the Poisson boundary (which is the product space in the measure-preserving case) must have certain structure.  We generalize the induced action \cite{SZ} to the stationary setting (Definition \ref{D:statind}) to apply the factor theorem to actions of lattices.\ifthenelse{\equal{\includeappendix}{true}}{Similarly, we generalize the intermediate factor theorem for commensurators of \cite{CP17} to the stationary case: Theorem \ref{T:iftdense}.}{\relax}

\subsection*{Projecting Actions to the Ambient Group}

The other main tool we introduce is a technique for projecting the action of a lattice or commensurator to the ambient group: Theorems \ref{T:projectinglatt} and \ref{T:projectingcomm}.  These techniques give a very general method for relating the stabilizers of the action of a lattice or commensurator to the stabilizers of an action of the ambient group (on some other space).  The technique is summarized as: if $\Lambda$ is a dense commensurator of a lattice in a locally compact second countable group $G$ and $\eta$ is a stationary random subgroup of $\Lambda$ then there exists a stationary random subgroup of $G$ via the closure map on subgroups, see Corollary \ref{C:projectingSRS}.

\subsection*{Related Work}

Boutonnet and Houdayer \cite{BH19} proved an operator-algebraic statement that all stationary characters on a lattice in a connected semisimple group, all simple factors of higher-rank, are trivial or arise from finite-index subrepresentations; in particular their result implies that stationary actions of such lattices have finite stabilizers or finite index stabilizers.
Our result covers a much larger class of lattices, including those where only one simple factor is higher-rank and those with $p$-adic components.

Extremely recently, \cite{BBHP20}, Bader, Boutonnet, Houdayer and Peterson, independently obtained the result on stationary actions of general lattices as a consequence of an incredibly strong operator-algebraic rigidity statement for lattices.
Our result on stationary actions of lattices in products of groups, announced prior to their work (e.g.~\cite{vanderbilttalk}), relies on dynamical methods very different from their algebraic methods.

\subsection*{Acknowledgments}

The author would like to thank R.~Boutonnet for numerous helpful discussions, and particularly for pointing out some flaws in early drafts.

\section{Lattices, Commensurators and Stationary Actions}

Let $G$ be a locally compact second countable group and $\mu \in P(G)$ an admissible probability measure (nonsingular with respect to Haar measure and with support generating the group).  $(G,\mu) \actson (X,\nu)$ means a quasi-invariant action on a probability space with $\mu * \nu = \nu$.

\subsection{Stationary Actions of Lattices}

\begin{definition}\label{D:strirr}
Let $\Gamma < G = G_{1} \times \cdots \times G_{k}$ be a lattice in a product of locally compact second countable groups.

$\Gamma$ is strongly irreducible when for every proper subproduct $G_{0}$, the projection map $\mathrm{proj}_{G_{0}} : G \to G_{0}$ has the properties that the image of $\Gamma$ is dense and the map is faithful on $\Gamma$.
\end{definition}

\begin{definition}\label{D:statactlattice}
Let $\Gamma < G = G_{1} \times \cdots \times G_{k}$ be a lattice in a product of locally compact second countable groups.

Let $\kappa = \kappa_{1} \times \cdots \times \kappa_{k}$ be an admissible probability measure on $G$.  When $G$ is a semisimple group, we take $\kappa_{j}$ to be $K_{j}$-invariant where $K_{j}$ is a maximal compact subgroup if $G_{j}$ is connected and $K_{j}$ is a compact open subgroup if $G_{j}$ is totally disconnected.

Let $\mu$ be the admissible probability measure on $\Gamma$ such that the $\Gamma$ action on the $(G,\kappa)$-Poisson boundary is $\mu$-stationary \cite{Ma91}.

A \textbf{stationary action} of $\Gamma$ is an action which is $\mu$-stationary for some $\mu$ coming from a product measure $\kappa$ on $G$.
\end{definition}

\subsection{Stationary Actions of Commensurators}

\begin{definition}
Let $\Gamma < G$ be a lattice in a locally compact second countable group.  A countable group $\Lambda < G$ is a \textbf{commensurator} of $\Gamma$ when $\Gamma \subseteq \Lambda$ and for all $\lambda \in \Lambda$ the group $\Gamma \cap \lambda \Gamma \lambda^{-1}$ has finite index in both $\Gamma$ and $\lambda \Gamma \lambda^{-1}$.
\end{definition}

We first establish that dense commensurators are, in a suitable sense, lattices in their own right:

\begin{theorem}\label{T:cl}
Let $\Gamma < G$ be a strongly irreducible lattice in a semisimple group $G$ with trivial center of higher-rank (meaning $G$ is either a simple higher-rank connected Lie group or $G$ is a product of at least two simple factors).

Let $\Lambda$ be a dense commensurator of $\Gamma$.

Then there exists a locally compact totally disconnected nondiscrete group $H$ such that $\Lambda$ can be embedded as a strongly irreducible lattice in $G \times H$.
\end{theorem}
\begin{proof}
The group $H$ will be the relative profinite completion $\rpf{\Lambda}{\Gamma}$ (see e.g.~\cite{CP17} Section 6 for details, we indicate here the facts needed for the proof).  Define the map $\tau : \Lambda \to \mathrm{Symm}(\Lambda / \Gamma)$ sending $\lambda$ to its action on $\Lambda / \Gamma$ as a symmetry of $\Lambda / \Gamma$.  Theorem 6.3 in \cite{CP17} states that $H = \overline{\tau(\Lambda)}$ is a locally compact totally disconnected group and $K = \overline{\tau(\Gamma)}$ is a compact open subgroup.

Let $N = \mathrm{ker}(\tau)$ be the kernel of $\tau$.  Then $N \cap \Gamma \normal \Gamma$ and $N \normal \Lambda$.  By Margulis' Normal Subgroup Theorem, $N \cap \Gamma$ is either finite or has finite index in $\Gamma$.  Since $G$ is center-free, if $N$ is finite it is trivial.  Suppose it has finite index.  As $N \normal \Lambda$ then $\overline{N} \normal \overline{\Lambda} = G$.  The only normal subgroups of $G$ are finite or are proper subproducts.  Since $N$ contains a finite index subgroup of $\Gamma$, it is infinite, but then $N$ must be contained in a proper subproduct so a finite index subgroup of $\Gamma$ would be contained in a proper subproduct, contradicting that $\Gamma$ is strongly irreducible.  So the kernel of $\tau$ is trivial.

Proposition 6.1.2 \cite{CP17} states that $H$ will be discrete if and only if $\Lambda$ normalizes a finite index subgroup of $\Gamma$.  If $\Gamma_{0}$ is finite index in $\Gamma$ and $\Gamma_{0} \normal \Lambda$ then $\Gamma_{0} \normal G$ as $\Gamma_{0}$ is discrete and $\Lambda$ is dense, but this is impossible.  So $H$ is nondiscrete.

Clearly we can embed $\Lambda \to G \times H$ diagonally and faithfully (as the kernel of $\tau$ is trivial).  Likewise we can embed $\Gamma$ diagonally, and we can identify both with their images.

Let $U$ be an open neighborhood of the identity in $G$ such that $\Gamma \cap U = \{ e \}$.  Then $\Lambda \cap U \times K = \Gamma \cap U \times K = \{ e \}$ .  So $\Lambda$ is discrete in $G \times H$.  Let $F$ be a fundamental domain for $\Gamma$ in $G$: $\mathrm{Haar}_{G}(F) < \infty$ and $F\Gamma = G$.  Then $(F \times K)\Lambda = G \times H$ (as the projection to $H$ of $\Lambda$ is dense) and $\mathrm{Haar}_{G\times H}(F \times K) < \infty$ since $K$ is compact.  Hence $\Lambda$ is a lattice in $G \times H$.

By construction, $\Lambda$ projects densely to $H$.  As $\Lambda$ is dense in $G$ by hypothesis, it projects densely to $G$ from $G \times H$.  For any proper subproduct $G_{0}$ of $G$, $\Lambda$ projects densely to $G_{0}$ since $\Gamma$ does.  Since $\Gamma$ projects densely to $G_{0} \times K$, we have that $\Lambda$ projects densely to $G_{0} \times H$.  Thus $\Lambda$ is strongly irreducible.
\end{proof}

This allows us to define:

\begin{definition}
Let $\Gamma < G$ be a lattice and $\Lambda < G$ a dense commensurator.  Let $\kappa = \kappa_{G} \times \kappa_{H}$ be an admissible probability measure on $G \times H$.  Then there is a probability measure $\mu$ on $\Lambda$ so that the $\Lambda$ action on the $(G\times H,\kappa)$-Poisson boundary is $\mu$-stationary.

An action of $\Lambda$ is \textbf{stationary} when it is stationary for some $\mu$ coming from such a $\kappa$ (i.e.~when it is a stationary action if we treat $\Lambda$ as a lattice).
\end{definition}

We also need the converse of the above theorem; lattices can be treated as commensurators:

\begin{proposition}\label{P:relpro}
Let $\Lambda < G \times H$ be a strongly irreducible lattice, $G$ and $H$ locally compact second countable groups with $H$ totally disconnected and $K < H$ a compact open subgroup.

Let $\Lambda \actson (X,\nu)$ be a stationary action.

Set $\Gamma = \mathrm{proj}_{G}~(\Lambda \cap G \times K)$.

Then $\Gamma$ is a strongly irreducible lattice in $G$ and $\Gamma \actson (X,\nu)$ is a stationary action.
\end{proposition}
\begin{proof}
Let $L = \Lambda \cap G \times K$.  Then $L$ is discrete since $K$ is open.  As $K$ is compact, $L$ has finite covolume in $G \times H$.  Since $K$ is compact, $\Gamma = \mathrm{proj}_{G}~L$ is discrete hence is a lattice in $G$.

Let $\kappa = \kappa_{G} \times \kappa_{H}$ be the admissible measure on $G \times H$ and $\mu_{\Lambda}$ the probability measure on $\Lambda$ such that the $(G\times H,\kappa)$-Poisson boundary is $\mu_{\Lambda}$-stationary under the $\Lambda$-action and $\Lambda \actson (X,\nu)$ is $\mu_{\Lambda}$-stationary.  

Since $\kappa$ is a product measure, $PB(G\times H,\kappa_{G}\times\kappa_{H}) = (B_{G},\beta_{G}) \times (B_{H},\beta_{H})$ is the product of the $(G,\kappa_{G})$- and $(H,\kappa_{H})$-boundaries.  The boundary map $b : B_{G} \times B_{H} \to P(X)$ has the property that $\mathrm{bar}~b_{*}(\beta_{G} \times \beta_{H}) = \nu$.

Define $\mu_{\Gamma}$ to be the probability measure on $\Gamma$ so that $(B_{G},\beta_{G})$ is $\mu_{\Gamma}$-stationary under the $\Gamma$-action.

Treating $\Gamma < \Lambda < G \times H$, we have that $\mathrm{proj}_{H}~\Gamma$ leaves $\beta_{H}$ invariant since $\kappa_{H}$ is chosen to be $K$-invariant.  Since $\mu_{\Gamma} * \beta_{G} = \beta_{G}$ this means that $\mu_{\Gamma} * \beta_{G} \times \beta_{H} = \beta_{G} \times \beta_{H}$.

Since the boundary map $B$ is $\Lambda$-equivariant, it is $\Gamma$-equivariant and so we conclude that
\[
\mu_{\Gamma} * \nu = \mu_{\Gamma} * \mathrm{bar}~b_{*}(\beta_{G} \times \beta_{H})
=  \mathrm{bar}~b_{*}(\mu_{\Gamma} * \beta_{G} \times \beta_{H})
=  \mathrm{bar}~b_{*}(\beta_{G} \times \beta_{H}) = \nu
\]
meaning the action is stationary for $\Gamma$.
\end{proof}

\section{Stationary Factor Theorems}

The intermediate factor theorems of Stuck-Zimmer \cite{SZ} and Bader-Shalom \cite{BS}, as well as those of the author \cite{CP17}, \cite{Cr17}, all make assertions about the structure of an intermediate factor $A$ between a measure-preserving $G$-space $(X,\nu)$ and the product of it with the Poisson boundary $(B \times X, \beta \times \nu)$.  We now establish such a factor theorem for stationary actions.

\subsection{The Stationary Joining}

For a stationary $G$-space $(X,\nu)$, the product space $(B \times X, \beta \times \nu)$ will not in general be stationary (and indeed will only be when $X$ is measure-preserving).  Our factor theorem employs the stationary joining of Furstenberg-Glasner \cite{FG} in its place.  The reader is referred to Glasner \cite{glasner} for a detailed exposition on joinings.
 
\begin{definition}[\cite{FG}]
Let $G$ be a locally compact second countable group and $\mu$ an admissible measure on $G$.  Let $(X,\nu)$ and $(Y,\eta)$ be $(G,\mu)$-spaces.  Let $\nu_{\omega}$ and $\eta_{\omega}$, for $\omega \in G^{\mathbb{N}}$, denote the conditional measures (see e.g.~\cite{BS} Theorem 2.10) which exist $\mu^{\mathbb{N}}$-almost surely.  Define the probability measure on $X \times Y$ by
$\rho = \int_{G^{\mathbb{N}}} \nu_{\omega} \times \eta_{\omega}~d\mu^{\mathbb{N}}(\omega)$.
Then $(X \times Y, \rho)$ is a joining of $(X,\nu)$ and $(Y,\eta)$ which is $\mu$-stationary.

This joining is called the \textbf{$\mu$-stationary joining} of the two systems and written
$(X,\nu) \curlyvee (Y,\eta)$.
\end{definition}

\subsection{The Invariant Products Functor}

To formulate the intermediate factor theorem for general stationary actions, we recall the invariant products functor, introduced in \cite{BS}:

\begin{definition}
Let $G = G_{1} \times \cdots \times G_{k}$ be a product of locally compact second countable groups.
For each $j$ write $\check{G}_{j} = \prod_{i \ne j} G_{i}$ for the subproduct excluding $G_{j}$.

For a $G$-space $(Y,\eta)$ write $(Y_{j},\eta_{j})$ for the $\check{G}_{j}$-ergodic components.  When $G \actson (Y,\eta)$ is ergodic, $G_{j} \actson (Y_{j},\eta_{j})$ is ergodic.

The $G$-map $(Y,\eta) \to (Y_{1},\eta_{1}) \times \cdots \times (Y_{k},\eta_{k})$ is the \textbf{invariant products functor}.
\end{definition}

\subsection{The Stationary Intermediate Factor Theorem for Product Groups}

\begin{theorem}\label{T:factorprod}
Let $G = G_{1} \times \cdots \times G_{k}$ be a product of at least two locally compact second countable groups and $\mu = \mu_{1} \times \cdots \times \mu_{k}$ be an admissible probability measure on $G$.

Let $(X,\nu)$ be an ergodic $(G,\mu)$-space.  Let $(B,\beta)$ be the Poisson boundary for $(G,\mu)$.

Let $(A,\alpha)$ be an intermediate factor:
\[
(B,\beta) \curlyvee (X,\nu) \to (A,\alpha) \to (X,\nu)
\]
where the maps are $G$-maps that compose to the natural projection map.

Then $A$ is isomorphic to the relative independent joining of $A_{1} \times \cdots \times A_{k}$ and $X$ over their common factor $X_{1} \times \cdots \times X_{k}$ (where the $A_{1}\times\cdots\times A_{k}$ and $X_{1}\times\cdots\times X_{k}$ are the invariant products functor on $A$ and $X$).
\end{theorem}
\begin{proof}
The invariants product functor (see \cite{Cr17} Section 2.12) mapping a $G$-space to the product of the spaces of $G_{j}$-ergodic components is relatively measure-preserving for stationary actions (\cite{Cr17} Proposition 2.12.1; see also \cite{BS} Proposition 1.10).

The $\tilde{G}_{j}$ ergodic components of $B \curlyvee X$ are $B_{j} \curlyvee X_{j}$ (the reader may verify this straightforward fact).  Thus we have the commuting diagram of $G$-maps
\begin{diagram}
B \curlyvee X								&\rTo			&A				&\rTo		&X\\
\dTo									&				&\dTo			&					&\dTo\\
(B_{1} \curlyvee X_{1}) \times \cdots \times (B_{k} \curlyvee X_{k})	&\rTo	& A_{1} \times \cdots \times A_{k}	&\rTo &X_{1} \times \cdots \times X_{k}
\end{diagram}

Since $B$ is a contractive space, the map $B \curlyvee X \to X$ is relatively contractive hence so is the map $A \to X$.  As the downward maps are relatively measure-preserving, Theorem 2.41 in \cite {Cr17} implies that $A$ is isomorphic to the relative independent joining of $A_{1}\times\cdots\times A_{k}$ and $X$ over $X_{1}\times\cdots \times X_{k}$.
\end{proof}

\begin{corollary}\label{C:denseprojs}
Let $G = G_{1} \times \cdots \times G_{k}$ be a product of locally compact second countable groups and $\mu = \mu_{1} \times \cdots \times \mu_{k}$ be an admissible probability measure on $G$.

Let $(X,\nu)$ be an ergodic $(G,\mu)$-space such that $\mathrm{proj}_{j}~\stab(x)$ is dense in $G_{j}$ for all $j$ almost everywhere.
Then $(X,\nu)$ is measure-preserving and weakly amenable.
\end{corollary}
\begin{proof}
The map $s_{j} : X \to S(G_{j})$ by $x \mapsto \overline{\mathrm{proj}_{G_{j}}~\stab(x)}$ sending each $x$ to a closed subgroup of $G_{j}$ gives rise to a random subgroup of $G_{j}$; realizing that stationary random subgroup, via Theorem 3.3 \cite{Cr17}, as the stabilizers of an action $G_{j} \actson (Z_{j},\zeta_{j})$ gives us the $(G,\mu)$-stationary space  $Z = Z_{1} \times \cdots \times Z_{k}$.

By construction, the stabilizers of $Z$ contain $\overline{\mathrm{proj}_{1}~\stab(x)} \times \cdots \times \overline{\mathrm{proj}_{k}~\stab(x)} = G_{1} \times \cdots \times G_{k}$.  As $Z$ is ergodic (since $X$ is), this means $Z$ is the trivial space.

Theorem 4.12 in \cite{Cr17} says that $X \to Z$ is a relatively measure-preserving extension (since $X$ is stationary and $Z = Z_{1} \times \cdots \times Z_{k}$ is the product random subgroups functor of \cite{Cr17} applied to $X$) hence we conclude that $(X,\nu)$ is in fact measure-preserving.

Now consider the factor theorem when $A$ is an affine space over $(X,\nu)$ following the approach pioneered by Stuck and Zimmer \cite{SZ} and used in \cite{BS}, \cite{CP17}, \cite{Cr17}, etc.  The reader is referred to \cite{Cr17} Section 2 for details on affine spaces and their relation to strong and weak amenability of actions.

For any affine space $A$ over $(X,\nu)$, as the action of $G$ on the Poisson boundary is strongly amenable (\cite{Zim84} Section 4.3), there are $G$-maps
$B \curlyvee X \to A \to X$
which compose to the projection map.  In particular, the pushforward of the stationary measure on $B \times X$ endows $A$ with the structure of a $(G,\mu)$-space.

We consider the case when $A$ is orbital over $X$ (has the same stabilizer subgroups).  Since the projection of the stabilizers are dense, the stabilizers of $X_{j}$ are $G_{j}$ almost surely meaning that each $X_{j}$ is trivial.  Therefore $A$ is isomorphic to the independent joining of $X$ and $A_{1} \times \cdots \times A_{k}$.

Let $a \in A$.  Then $\stab(a) = \stab(x)$ for some $x$ mapped to from $a$ so almost every $a$ has stabilizer that projects densely to each $G_{j}$.  As each $A_{j}$ is a $G_{j}$-space, this means that the stabilizers for $A_{j}$ are $G_{j}$ almost everywhere.  So the $A_{j}$ are all trivial meaning that $A$ is isomorphic to $X$.

As this holds for all affine orbital $A$, it follows that $X$ is weakly amenable.
\end{proof}

\section{Inducing Stationary Actions}

Let $\Gamma$ be an irreducible lattice in a locally compact second countable group $G$.  Let $F$ be a fundamental domain for $G / \Gamma$ with associated cocycle $\alpha : G \times F \to \Gamma$ given by $gf\alpha(g,f) \ in F$.  Write $m$ for the Haar measure on $G$ restricted to $F$ and normalized to be a probability measure.

\subsection{The Classical Induced Action}

The classical induced action, due to Zimmer (e.g.~\cite{SZ}), constructs a measure-preserving $G$-action from a measure-preserving $\Gamma$-action as follows: define $G \actson F \times X$ by
\[
g \cdot (f,x) = (gf\alpha(g,f),\alpha(g,f)^{-1}x)
\]
and then the measure $m \times \nu$ is preserved by the $G$-action.

Much of the issue in finding a generalization to stationary actions lies in the fact that there is exactly one $G$-quasi-invariant measure on $F \times X$ which projects to $\nu$ on $X$ and this measure is not stationary (unless $\nu$ is preserved by $\Gamma$) so our induced action will not have the property that its projection to $X$ is $\nu$.  However, the projection will be in the same measure class.

\subsection{The Stationary Induced Action}

Let $\kappa$ be an admissible probability measure on $G$ and let $(B,\beta)$ be the Poisson boundary for $(G,\kappa)$.  Write $\mu$ for the probability measure on $\Gamma$ so that $(B,\beta)$ is $\mu$-stationary.  Since the action of $\Gamma$ on $(B,\beta)$ is strongly amenable (as the action of $G$ is and $\Gamma$ is closed in $G$) there exists the boundary $\Gamma$-map $\pi : B \to P(X)$ (Zimmer \cite{Zim84}; see e.g.~\cite{BS} Theorem 2.14).

Let $\mathrm{bar} : P(P(X)) \to P(X)$ be the barycenter map: $\mathrm{bar}~\eta = \int_{P(X)} \nu~d\eta(\nu)$ (see Furman \cite{furman}).  Since $\pi_{*} : P(B) \to P(P(X))$ is the pushforward map, we have a $\Gamma$-map $\tilde{\pi} = \mathrm{bar} \circ \pi_{*} : P(B) \to P(X)$ such that $\tilde{\pi}(\beta) = \nu$.

Let $(X,\nu)$ be a $(\Gamma,\mu)$-space.  Define the pointwise action $G \actson F \times X$ as in the classical case.  Define the probability measure $\rho$ on $F \times X$ by
\[
\rho = \int_{F} \delta_{f} \times \tilde{\pi}(f^{-1}\beta)~dm(f)
\]
Observe that, since $m$ is $G$-invariant,
\begin{align*}
\kappa * \rho
&= \int_{G} \int_{F} g \cdot (\delta_{f} \times \tilde{\pi}(f^{-1}\beta))~dm(f)~d\kappa(g) \\
&= \int_{G} \int_{F} \delta_{gf\alpha(g,f)} \times \alpha(g,f)^{-1}\tilde{\pi}(f^{-1}\beta)~dm(f)~d\kappa(g) \\
&= \int_{G} \int_{F} \delta_{gf\alpha(g,f)} \times \tilde{\pi}(\alpha(g,f)^{-1}f^{-1}g^{-1}~g\beta)~dm(f)~d\kappa(g) \\
&= \int_{G} \int_{F} \delta_{f_{0}} \times \tilde{\pi}(f_{0}^{-1}~g\beta)~dm(f_{0})~d\kappa(g) \\
&= \int_{F} \delta_{f_{0}} \times \tilde{\pi}(f_{0}^{-1}~\kappa * \beta)~dm(f_{0}) \\
&=  \int_{F} \delta_{f_{0}} \times \tilde{\pi}(f_{0}^{-1}\beta)~dm(f_{0}) 
= \rho
\end{align*}
meaning that $(F \times X, \rho)$ is a $(G,\kappa)$-space, i.e.~$\kappa * \rho = \rho$.

\begin{definition}\label{D:statind}
The $(G,\kappa)$-space just defined is the \textbf{induced stationary action} from the $(\Gamma,\mu)$-space $(X,\nu)$.
We will write $G \times_{\Gamma} X$ for this space.
\end{definition}

\subsection{Properties of the Induced Action}

\begin{proposition}
The measure on the induced stationary space projects to $F$ as $m$ and projects to $X$ as a measure in the same class as $\nu$ (when $\nu$ is measure-preserving, it projects to $\nu$ itself).
\end{proposition}
\begin{proof}
Observe that 
$\mathrm{proj}_{X}~\rho = \int_{F} \tilde{\pi}(f^{-1}\beta)~dm(f)
= \tilde{\pi}(\check{m}*\beta)$
where $\check{m}$ is the symmetric opposite of $m$.  So $\mathrm{proj}_{X}~\rho$ is in the same measure class as $\nu$ since $\check{m}*\beta$ is in the same measure class as $\beta$.  That $\mathrm{proj}_{F}~\rho = m$ is immediate.
\end{proof}

\begin{proposition}\label{P:indwa}
The induced stationary action is weakly amenable if and only if the action on $(X,\nu)$ is weakly amenable.
\end{proposition}
\begin{proof}
This is Proposition 2.6.1 in \cite{Cr17}.  While that proposition is stated as holding for measure-preserving actions, the result it relies on, stated as Theorem 2.25 \cite{Cr17}, is actually a result of Zimmer \cite{Zi77} which holds for quasi-invariant actions.
\end{proof}

\begin{proposition}\label{P:indmp}
If the induced stationary action is measure-preserving so is $(X,\nu)$.
\end{proposition}
\begin{proof}
Since
$g \cdot \rho = \int \delta_{f_{0}} \times \tilde{\pi}(f_{0}^{-1}g\beta)~dm(f_{0})$
the induced action being measure-preserving implies that
$\tilde{\pi}(f^{-1}\beta) = \tilde{\pi}(f^{-1}g\beta)$
for all $f \in F$ and $g \in G$ meaning that $\tilde{\pi}(\beta)$ is $\Gamma$-invariant.
\end{proof}

\subsection{The Stationary Intermediate Factor Theorem for Lattices}

\begin{corollary}\label{C:latticedenseprojs}
Let $\Gamma < G = G_{1} \times \cdots \times G_{k}$ be a strongly irreducible lattice in a product of at least two locally compact second countable groups.

Let $\Gamma \actson (X,\nu)$ be an ergodic stationary action such that $\mathrm{proj}_{j}~\stab(x)$ is dense in $G_{j}$ for all $j$ almost everywhere.  Then $(X,\nu)$ is measure-preserving and weakly amenable.
\end{corollary}
\begin{proof}
Let $(Y,\eta)$ be the induced stationary action.  Since
$\stab(f,x) = f\stab(x)f^{-1}$,
we have that for almost every $y$ and all $j$,
$\overline{\mathrm{proj}_{j}~\stab(y)} = G_{j}$.
Then $Y$ is measure-preserving by Corollary \ref{C:denseprojs} hence $X$ is by Proposition \ref{P:indmp}.
\end{proof}

\section{Projecting Actions to the Ambient Group}

\begin{theorem}\label{T:projectingcomm}
Let $\Gamma < G$ be a lattice in a locally compact second countable group and $\Lambda < G$ a dense commensurator of $\Gamma$.

Let $\Lambda \actson (X,\nu)$ be an ergodic stationary action.

Let $s : X \to S(G)$ be a $\Lambda$-equivariant map.

Then there exists an ergodic stationary $G$-space $(Y,\eta)$ such that $\stab_{*}\eta = s_{*}\nu$.
\end{theorem}
\begin{proof}
The measure-preserving case, used in \cite{CP17}, is an easy consequence of density.  The stationary case requires more care.  Treat $(X,\nu)$ as a $\Gamma$-space and let $\pi : (B,\beta) \to (P(X),\pi_{*}\beta)$ be the boundary $\Gamma$-map where $(B,\beta)$ is the $G$-Poisson boundary.  Consider the $\Gamma$-map $s_{*} \circ \pi : B \to P(S(G))$ which has the property that $\mathrm{bar}~s_{**}\pi_{*}(\beta) = s_{*}\nu$.

Since $\Lambda$ acts on $(B,\beta)$ and $(P(S(G)),s_{*}\pi_{*}\beta)$, and since $(B,\beta)$ is a contractive $\Gamma$-space, Theorem 2.8 \cite{CS14} states that $s_{*} \circ \pi$ is in fact a $\Lambda$-equivariant map and there is a $G$-space $(Z,\zeta)$ that is $\Lambda$-isomorphic to $(P(S(G)),s_{**}\pi_{*}\beta)$ via a $\Lambda$-map $\varphi$.

As $\Lambda$ is dense and $G \actson P(S(G))$ continuously, then $s_{**}\pi_{*}\beta$ is in fact $G$-stationary as for every $g \in G$ there exists $\lambda_{n} \in \Lambda$ with $\lambda_{n} \to g$ and so $gs_{**}\pi_{*}\beta = \lim \lambda_{n} s_{**}\pi_{*}\beta = \lim \lambda_{n}\varphi_{*}\zeta = \lim \varphi_{*}(\lambda_{n}\zeta) = \varphi_{*}(g\zeta)$.

As the barycenter map is $G$-equivariant, then $\mathrm{bar}~s_{**}\pi_{*}\beta$ is $G$-stationary and in particular there is a quasi-invariant action of $G$ on $(S(G),\mathrm{bar}~s_{**}\pi_{*}\beta)$.  Theorem 3.3 \cite{Cr17} then gives an action $G \actson (Y,\eta)$ such that $\stab_{*}\eta = \mathrm{bar}~s_{**}\pi_{*}\beta = s_{*}\nu$ which may be taken to be ergodic since $\Lambda \actson (X,\nu)$ is ergodic.
\end{proof}

\begin{corollary}\label{C:projectingSRS}
Let $\Gamma < G$ be a lattice in a locally compact second countable group and $\Lambda < G$ a dense commensurator of $\Gamma$.  Every stationary random subgroup of $\Lambda$ becomes a stationary random subgroup of $G$ under the closure map.
\end{corollary}

\begin{theorem}\label{T:projectinglatt}
Let $G = \prod_{j} G_{j}$ be a product of at least two simple locally compact second countable groups and let $\Gamma < G$ be a strongly irreducible lattice.

Let $\Gamma \actson (X,\nu)$ be an ergodic stationary action.

Let $s : X \to S(G_{0})$ be a $\Gamma$-equivariant map to the space of closed subgroups of some proper subproduct $G_{0} \normal G$.

Then there exists an ergodic $G_{0}$-space $(Y,\eta)$ such that $\stab_{*}\eta = s_{*}\nu$.
\end{theorem}

\begin{lemma}[Lemma 5.1 \cite{peterson}]\label{L:pet}
Let $\Gamma < G \times H$ be an irreducible lattice, $U \subseteq G$ an open neighborhood of the identity, $(B_{H},\beta)$ the $(H,\kappa)$-Poisson boundary, $E \subseteq B_{H}$ a positive measure set and $\epsilon > 0$.

There exists $\gamma \in \Gamma$ such that $\mathrm{proj}_{G}~\gamma \in U^{-1}U$ and $\beta((\mathrm{proj}_{H}~\gamma)E) > 1 - \epsilon$.
\end{lemma}
\begin{proof}
Since $\Gamma$ is irreducible, $H$ acts ergodically on $((G \times H)/\Gamma,m)$ where $m$ is the unique invariant probability measure.  Let $K \subseteq H$ be compact with nonempty interior.  Then $m((U \times K)\Gamma) > 0$.  The random ergodic theorem \cite{kakutani} tells us for $m$-almost every $z \in (G \times H)/\Gamma$ and $\kappa^{\mathbb{N}}$-almost every $(\omega_{n}) \in H^{\mathbb{N}}$ that $\frac{1}{N}\sum_{n=1}^{N}\bbone_{(U \times K)\Gamma}(\omega_{n}^{-1}\cdots\omega_{1}^{-1}z) \to m((U \times K)\Gamma) > 0$.  So there exists $(u,k) \in U \times K$ such that $\omega_{n}^{-1}\cdots\omega_{1}^{-1}(u,k) \in (U \times K)\Gamma$ infinitely often for a.e.~$(\omega_{n})$.

The conditional measures $\beta_{(\omega_{n})} = \lim \omega_{1}\cdots\omega_{n}\beta$ are point masses almost surely and $\beta = \int \beta_{(\omega_{n})}~d\kappa^{\mathbb{N}}((\omega_{n}))$ so, as $\beta(kE) > 0$ as $\beta(E) > 0$, there exists $(\omega_{n}) \in H^{\mathbb{N}}$ with $\omega_{n}^{-1}\cdots\omega_{1}^{-1}(u,k) \in (U \times K)\Gamma$ infinitely often and $\beta_{(\omega_{n})}(kE) = 1$.  So there exists $\{ n_{j} \}$ so that $\omega_{n_{j}}^{-1}\cdots\omega_{1}^{-1}(u,k) = (u_{j},k_{j})\gamma_{j}$ for some $u_{j} \in U$, $k_{j} \in K$, $\gamma_{j} \in \Gamma$ and $\beta(\omega_{n_{j}}^{-1}\cdots\omega_{1}^{-1}kE) \to 1$.

As $\gamma_{j} = (u_{j},k_{j})^{-1} \omega_{n_{j}}^{-1}\cdots\omega_{1}^{-1} (u,k) = (u_{j}^{-1}u, k_{j}^{-1}\omega_{n_{j}}^{-1}\cdots\omega_{1}^{-1}k)$, we have $\mathrm{proj}_{G}~\gamma_{j} \in U^{-1}U$ and $\beta(k_{j}(\mathrm{proj}_{H}~\gamma_{j})E) \to 1$.  Since $K$ is compact, there is $\{ j_{\ell} \}$ so that $k_{j_{\ell}} \to k_{\infty} \in K$.  We know the sets $E_{\ell} = k_{j_{\ell}}(\mathrm{proj}_{H}~\gamma_{j_{\ell}})E$ have $\beta(E_{\ell}) \to 1$ and since $k_{\infty}\beta$ is in the same measure class as $\beta$ then $\beta(k_{j_{\ell}}^{-1}E_{\ell}) \to 1$.  So there exists $\ell$ such that $1 - \epsilon < \beta(k_{j_{\ell}}^{-1}E_{\ell}) = \beta((\mathrm{proj}_{H}~\gamma_{j_{\ell}})E)$.
\end{proof}

\begin{proof}[Proof of Theorem \ref{T:projectinglatt}]
Let $(B,\beta)$ be the Poisson boundary for $G$.  Let $\pi : (B,\beta) \to (P(X),\pi_{*}\beta)$ be the boundary map and consider the $\Gamma$-map $s_{*} \circ \pi : B \to P(S(G_{0}))$. Note the barycenter is $\mathrm{bar}~s_{**}\pi_{*}\beta = s_{*}\nu$.

Write $G = G_{0} \times \tilde{G}$ where $\tilde{G}$ is the simple factors of $G$ not in $G_{0}$.  Then $(B,\beta) = (B_{0},\beta_{0}) \times (\tilde{B},\tilde{\beta})$ is the product of the Poisson boundaries of $G$ and $G_{0}$ (we treat each as a $G$-space by letting $\tilde{G}$ act trivially on $B_{0}$ and $G_{0}$ act trivially on $\tilde{B}$).

The Intermediate Factor Theorem for products (Theorem 1.7 \cite{BS}) tells us $(P(S(G_{0})),s_{**}\pi_{*}\nu)$ is $\Gamma$-isomorphic to a $G$-space $(Z,\zeta)$ via a $\Gamma$-map $\varphi$ and that there is a $G$-map $p : (B,\beta) \to (Z,\zeta)$ such that $p = \varphi \circ s_{*}\pi$.  Moreover, $(Z,\zeta) = (Z_{0},\zeta_{0}) \times (\tilde{Z},\tilde{\zeta})$ where $G_{0}$ acts trivially on $\tilde{Z}$ and $p$ splits as $p_{0} : (B_{0},\beta_{0}) \to (Z_{0},\zeta_{0})$ and $\tilde{p} : (\tilde{B},\tilde{\beta}) \to (\tilde{Z},\tilde{\zeta})$.

Write $\mathrm{proj}_{0} : \Gamma \to G_{0}$ for the (faithful) projection map.
If $\gamma_{n} \in \Gamma$ such that $\mathrm{proj}_{0}~\gamma_{n} \to e$ in $G_{0}$ then $s(\gamma_{n}x) \to s(x)$ as $G_{0}$ acts continuously on $S(G_{0})$.  Therefore $\gamma_{n}\tilde{\zeta} = \varphi_{*}\gamma_{n}s_{**}\pi_{*}\beta \to \varphi_{*}s_{**}\pi_{*}\beta = \tilde{\zeta}$ whenever $\mathrm{proj}_{0}~\gamma_{n} \to e$ in $G_{0}$.

Let $E$ be a positive measure subset of $\tilde{Z}$.  Lemma \ref{L:pet} gives $\gamma_{n} \in \Gamma$ such that $\mathrm{proj}_{0}~\gamma_{n} \to e$ in $G_{0}$ and $\tilde{\beta}(\gamma_{n}\tilde{p}^{-1}(E)) \to 1$ (taking $U_{n}$ to be a decreasing sequence of open neighborhoods of $e$ and $\epsilon$ to be $1/n$).
Then $\tilde{\zeta}(\gamma_{n}E) \to 1$ but also $\gamma_{n}^{-1}\tilde{\zeta} \to \tilde{\zeta}$ so we conclude that $\tilde{\zeta}(E) = 1$.  So $(\tilde{Z},\tilde{\zeta})$ is trivial, i.e.~$(P(S(G_{0})),s_{**}\pi_{*}\beta)$ is $\Gamma$-isomorphic to the $G_{0}$-space $(Z_{0},\zeta_{0})$.

This means that if $\mathrm{proj}_{0}~\gamma_{n} \to g_{0}$ in $G_{0}$ then $\gamma_{n}\zeta \to g_{0}\zeta$ which is in the same measure class as $\zeta$.  So $g_{0}s_{**}\pi_{*}\beta = g_{0}\varphi_{*}^{-1}\zeta = \lim \gamma_{n} \varphi_{*}^{-1}\zeta = \lim \varphi_{*}^{-1}\gamma_{n}\zeta = \varphi_{*}g_{0}\zeta$ is in the same class as $s_{**}\pi_{*}\beta$.

As $\Gamma$ projects densely to $G_{0}$, this means $g_{0}s_{**}\pi_{*}\beta$ is in the same measure class as $s_{**}\pi_{*}\beta$ for every $g_{0} \in G_{0}$, i.e.~that $s_{**}\pi_{*}\beta$, and therefore also $s_{*}\nu$, is $G_{0}$-quasi-invariant (and in fact stationary) so the result follows using Theorem 3.3 \cite{Cr17}.
\end{proof}

\section{Stabilizers of Lattices in Connected Groups}

\begin{theorem}\label{T:latticeactions}
Let $G = G_{1} \times \cdots \times G_{k}$ be a product of simple connected Lie groups, $k \geq 2$.  Let $\Gamma < G$ be a strongly irreducible lattice.

Let $\Gamma \actson (X,\nu)$ be an ergodic stationary action.

Then either $\stab_{\Gamma}(x)$ projects densely to each simple $G_{k}$ almost everywhere or $\stab_{\Gamma}(x)$ is finite almost everywhere.
\end{theorem}

The proof will occupy this section.  We begin with an observation of Vershik \cite{Vershik}:

\begin{proposition}\label{P:fixed}
Let $\Lambda$ be a countable group and $\Lambda \actson (X,\nu)$ a quasi-invariant action.

For $\lambda \in \Lambda$ and for a subgroup $Q < \Lambda$ write
\[
\mathrm{Fix}~\lambda = \{ x : \lambda x = x \} \quad\quad\quad\quad
\mathrm{Fix}~Q = \{ x : qx = x \text{ for all } q \in Q \}
\]

For almost every $x$ the stabilizer $\stab(x)$ satisfies:
$\stab(x) \subseteq \{ \lambda : \nu(\mathrm{Fix}~\lambda) > 0 \}$.
In particular, if $\nu(\mathrm{Fix}~\lambda) = 0$ for all $\lambda \ne e$ then the action is essentially free.
\end{proposition}
\begin{proof}
Write $L = \{ \lambda : \nu(\mathrm{Fix}~\lambda) > 0 \}$.  Let $Z = \bigcup_{\lambda \notin L} \mathrm{Fix}~\lambda$.  Then $Z$ is a countable union of measure zero sets hence is measure zero.  For all $x \notin Z$ we then have $\stab(x) \subseteq L$ since $x \notin \mathrm{Fix}~\lambda$ for $\lambda \notin L$ meaning that $\lambda x \ne x$ for $\lambda \notin L$.
\end{proof}

\subsection{Howe-Moore Groups}

Recall that a group is Howe-Moore when for every unitary representation without invariant vectors, the matrix coefficients vanish at infinity \cite{HM79}.

\begin{lemma}\label{L:HM}
Let $G$ be a locally compact second countable group with the Howe-Moore property and $G \actson (Y,\eta)$ a quasi-invariant action on a probability space.

If there exists a subgroup $Q$ with $\eta(\mathrm{Fix}~Q) > 0$ and $\overline{Q}$ noncompact then $\mathrm{Fix}~Q = \mathrm{Fix}~G$, i.e.~almost every point in $\mathrm{Fix}~Q$ is fixed by all of $G$.
\end{lemma}
\begin{proof}
First, observe that for any $f \in L^{2}(Y,\eta)$ and $q \in Q$,
\[
\int_{\mathrm{Fix}~Q} f(qy)~d\eta(y) = \int_{\mathrm{Fix}~Q} f(y)~d\eta(y)
\]
and therefore $\frac{dq\eta}{d\eta}(y) = 1$ for $y \in \mathrm{Fix}~Q$ and $q \in Q$.

Let $\pi$ be the unitary representation of $G$ on $L^{2}(Y,\eta)$: $\pi(g)f(y) = f(g^{-1}y) \sqrt{\frac{dg\eta}{d\eta}(y)}$.  Let $P$ be the orthogonal projection from $L^{2}(Y,\eta)$ to the closed $G$-invariant subspace $\mathcal{I}$ of $\pi$-invariant vectors.  Then $\pi$ restricted to $L^{2}(Y,\eta) \ominus \mathcal{I}$ has no invariant vectors.

Let $f$ be the characteristic function of some positive measure subset $E \subseteq \mathrm{Fix}~Q$.  Then $f - Pf \in L^{2}(Y,\eta) \ominus \mathcal{I}$ so, as $G$ has the Howe-Moore property,
\[
\langle \pi(g)(f - Pf), (f - Pf) \rangle \to 0 \quad\quad\quad\quad \text{as $g$ leaves compact sets}
\]

For $q \in Q$ and $y \in E \subseteq \mathrm{Fix}~Q$ we have that $\pi(q)f(y) = f(q^{-1}y)\sqrt{\frac{dq\eta}{d\eta}(y)} = 1 = f(y)$.  For $y \notin E$ we have that $qy \notin E$ (since $qy \in E$ gives $q^{-1}qy = qy$) so $\pi(q)f(y) = f(q^{-1}y)\sqrt{\frac{dq\eta}{d\eta}(y)} = 0 = f(y)$.

As $Pf$ is $\pi$-invariant then
$\langle \pi(q)(f - Pf),(f - Pf) \rangle = \| f - Pf \|^{2}$.  
As $\overline{Q}$ is noncompact this means that $\| f - Pf \| = 0$, i.e.~that $f$ is in fact $\pi$ invariant.

Suppose $E$ is not a $G$-invariant set.  Then there exists $g \in G$ with $\eta(gE \setminus E) > 0$.  Writing $\bbone_{E}$ for the characteristic function
\[
\bbone_{E}(y) = (\bbone_{E}(y))^{2} = (f(y))^{2} = \big{(}\pi(g)f(y)\big{)}^{2} = \bbone_{gE}(y) \frac{dg\eta}{d\eta}(y)
\]
Therefore for $y \in gE \setminus E$ we have $0 = \frac{dg\eta}{d\eta}(y)$ contradicting that the $G$ action is quasi-invariant.  Therefore $E$ is $G$-invariant as a set.

Since every subset of $\mathrm{Fix}~Q$ is a $G$-invariant set, $G$ fixes almost every point in it.
\end{proof}

\subsection{Connected Lie Groups}

\begin{lemma}\label{L:dim}
Let $G$ be a simple connected Lie group and $G \actson (Z,\zeta)$ be a quasi-invariant (not necessarily ergodic) action.  Let $C(z)$ be the connected component of $\stab(z)$.  Let $d \geq 0$ be a nonnegative integer.

If $\mathrm{dim}~C(z) \leq d$ almost everywhere then one of the following holds:
\begin{itemize}
\item $\stab(z) = G$ on a positive measure set
\item $\stab(z)$ is compact almost everywhere
\item $\stab(z)$ is discrete almost everywhere
\item the diagonal action $G \actson (Z \times Z, \zeta \times \zeta)$ has $\mathrm{dim}~C(z_{1},z_{2}) \leq d-1$ for almost every $z_{1},z_{2} \in Z$ which have noncompact and nondiscrete stabilizers
\end{itemize}
\end{lemma}
\begin{proof}
Suppose $\zeta \times \zeta ( \{ (x,y) \in Z \times Z : C(x) = C(y) \text{ are noncompact} \} > 0$.  Fubini's theorem tells us that for some $x$ there is a positive measure set of $y$ so that $\zeta(\{ x : C(x) = C(y) \}) > 0$.  This means $Q = C(y)$ is a noncompact closed subgroup of $G$ with $\zeta(\mathrm{Fix}~Q) > 0$.  Lemma \ref{L:HM} then says that $\mathrm{Fix}~Q = \mathrm{Fix}~G$ so in fact $C(x) = G$ on that positive measure set and we are in the first case.

Proceed now assuming we are not in the first case.  Let
\[
M = \{ z \in Z : \stab(z) \text{ is not compact and not discrete } \}
\]
For almost every $x,y \in M$ we have $C(x) \ne C(y)$ by the above and that $\mathrm{dim}~C(x) > 0$ and $\mathrm{dim}~C(y) > 0$ since they are nondiscrete.  Then $\mathrm{dim}~C(x) \cap C(y) < \mathrm{dim}~C(x) \leq d$.
\end{proof}

\begin{proposition}\label{P:cpt}
Let $G$ be a connected simple Lie group.  Let $G \actson (Y,\eta)$ be an ergodic nontrivial $G$-space.

Then there exists a positive integer $m$ such that the diagonal action $G \actson (Y^{m},\eta^{m})$ has the property that for almost every $z \in Y^{m}$ the stabilizer subgroup $\stab(z)$ is discrete or compact.
\end{proposition}
\begin{proof}
Write $C(y)$ for the connected component of $\stab(y)$.  Set $d = \sup~\mathrm{dim}~C(y) \leq \mathrm{dim}~G$.

Suppose $\eta(\{ y \in Y : \stab(y) \text{ is not compact and not discrete} \}) > 0$.  Apply Lemma \ref{L:dim} to conclude that either $\eta(\{ y  \in Y : \stab(y) = G \}) > 0$ or that $G \actson (Y^{2},\eta^{2})$ has the property that for $\eta\times\eta$-almost every $y_{2} \in Y^{2}$ it holds that $\stab(y_{2})$ is at least one of compact, discrete, all of $G$, or has dimension $\leq d-1$.  Repeating this process by applying Lemma \ref{L:dim} to $(Y^{2},\eta^{2})$, since $d$ is finite, we conclude that there is some finite number $t$ of applications so that for $m = 2^{t}$ we have that the diagonal action $G \actson (Y^{m},\eta^{m})$ has the property that almost every stabilizer is at least one of compact, discrete or all of $G$.

If $\eta^{m}(\{y_{m} \in Y^{m} : \stab(y_{m}) = G \}) > 0$ then, writing $y_{m} = (x_{1},\ldots,x_{m})$ for $x_{j} \in Y$, we have $G = \stab(y_{m}) = \stab(x_{1}) \cap \stab(x_{2}) \cap \ldots \cap \stab(x_{m})$ so $\stab(x) = G$ for a positive $\eta$-measure set of $x \in Y$.  By ergodicity then this holds almost everywhere but that contradicts that $(Y,\eta)$ is nontrivial.
\end{proof}

\begin{proposition}\label{P:discrete}
Let $G$ be a semisimple connected Lie group with trivial center and $\Lambda < G$ a countable subgroup and $G \actson (Y,\eta)$ a quasi-invariant action.

If the $G$-stabilizers are discrete almost everywhere then $\Lambda \actson (Y,\eta)$ is essentially free.
\end{proposition}
\begin{proof}
Suppose the $\Lambda$-action is not essentially free.  Then there exists $\lambda \in \Lambda$, $\lambda \ne e$, with $\nu(\mathrm{Fix}~\lambda) > 0$.  Since $G$ is connected and $\lambda \notin Z(G)$, the centralizer subgroup of $\lambda$ does not contain an open neighborhood the identity in $G$.  So there exists $g_{n} \to e$ in $G$ such that $g_{n}\lambda g_{n}^{-1} \ne \lambda$ for all $n$.  Note that, writing $E = \mathrm{Fix}~\lambda$, we have $\eta(g_{n}E \symdiff E) \to 0$.  Take a subsequence along which $\eta(g_{n}E \symdiff E) < 2^{-n-1} \eta(E)$.  Then
\[
\eta(E \cap \bigcap_{n} g_{n}E) = \eta(E) - \eta(E \symdiff \bigcap_{n} g_{n}E) \geq \eta(E) - \sum_{n=1}^{\infty} \eta(E \symdiff g_{n}E) > \frac{1}{2}\eta(E) > 0
\]
For $y \in E \cap (\cap_{n} g_{n}E) = \mathrm{Fix}~\lambda \cap \cap_{n} \mathrm{Fix}~g_{n}\lambda g_{n}^{-1}$ we have that $\lambda y = y$ and $g_{n}\lambda g_{n}^{-1} y = y$ for all $n$, hence $g_{n}\lambda g_{n}^{-1} \in \mathrm{stab}_{G}(y)$ and $\lambda \in \mathrm{stab}_{G}(y)$.  But $g_{n}\lambda g_{n}^{-1} \ne \lambda$ and $g_{n}\lambda g_{n}^{-1} \to \lambda$ contradicting that $\mathrm{stab}_{G}(y)$ is discrete almost everywhere.
\end{proof}

\subsection{Miscellany}

\begin{lemma}\label{L:topotrick}
Let $G$ be a locally compact second countable group and $K$ a compact group.  Let $Q < G \times K$ be a closed discrete subgroup.  Then the projection of $Q$ to $G$ is discrete in $G$.
\end{lemma}
\begin{proof}
Let $V$ be an open set in $G$ with compact closure.  Then $Q \cap \overline{V} \times K$ is finite since it is closed, hence compact, and discrete.  So $Q \cap V \times K$ is finite meaning that $(\mathrm{proj}_{G}~Q) \cap V$ is finite.  So for every $q \in \mathrm{proj}_{G}~Q$ we can find an open $U$ with $U \cap \mathrm{proj}_{G}~Q = \{ q \}$.
\end{proof}

\begin{lemma}\label{L:srssupp}
Let $H$ be a group and $L < H$ a subgroup.  If $\tau$ is a random subgroup of $H$ that is supported on subgroups of $L$ then $\tau$ is supported on subgroups of $N$ for some $N \normal H$ with $N \subseteq L$.
\end{lemma}
\begin{proof}
Since $\tau$ is supported on $S(L)$, for $\tau$-almost every $Q \in S(H)$ in fact $Q \subseteq L$.  As $\tau$ is $H$-quasi-invariant, for all $h$ and $\tau$-almost every $Q$ we have $hQh^{-1} \subseteq L$ as well.  Then $Q \subseteq \cap_{h} h^{-1}Lh$ for almost every $Q$.  Set $N = \cap_{h} h^{-1}Lh$ which is normal in $H$ and contained in $L$.  Then $Q \subseteq N$ almost everywhere as claimed.
\end{proof}

Recall that a group is \textbf{locally finite} when every finitely generated subgroup of it is finite.

\begin{proposition}\label{P:locfin}
Let $\Gamma$ be a virtually torsion-free countable discrete group and $\Gamma \actson (X,\nu)$ a quasi-invariant action with locally finite stabilizers almost everywhere.  Then the stabilizers are finite almost everywhere.
\end{proposition}
\begin{proof}
Since $\Gamma$ is virtually torsion-free, there is a finite upper bound $\ell$ for the length of every strictly increasing chain of finite subgroups.

As $\stab(x)$ is locally finite, if it is infinite then it contains an infinite abelian subgroup $A$ (Hall and Kulatilaka \cite{HK}).  But $A$ must be of the form $\oplus_{n=1}^{\infty} F_{n}$ for some nontrivial finite subgroups $F_{n}$ and then $A_{N} = \oplus_{n=1}^{N} F_{n}$ would be a strictly increasing chain of finite subgroups of infinite length.  Therefore $\stab(x)$ must in fact be finite.
\end{proof}

\subsection{Lattices and Projected Actions}

\begin{proposition}\label{P:nearly}
Let $G = G_{1} \times \cdots \times G_{k}$ be a product of simple connected Lie groups, $k \geq 2$.

Let $\Gamma < G$ be a strongly irreducible lattice.

Let $\Gamma \actson (X,\nu)$ be an ergodic stationary action.

Then either $\stab_{\Gamma}(x)$ is locally finite almost everywhere or $\stab_{\Gamma}(x)$ has dense projections to each simple $G_{j}$ almost everywhere.
\end{proposition}
\begin{proof}
For each $j \in \{ 1, 2, \ldots, k \}$, the projection map $\mathrm{proj}_{G_{j}} : G \to G_{j}$ has the property that the image of $\Gamma$ is dense in $G_{j}$ and the map is faithful on $\Gamma$.  Define the map $s_{j} : X \to S(G_{j})$ by $\overline{s_{j}}(x) = \overline{\mathrm{proj}_{G_{j}}~\stab_{\Gamma}(x)}$.  

By Theorem \ref{T:projectinglatt}, there is an ergodic $G_{j}$-space $(Y_{j},\eta_{j})$ such that $(\stab_{G_{j}})_{*}\eta_{j} = (s_{j})_{*}\nu$.

Define the sets
\[
J_{d} = \{ j \in \{ 1, \ldots, k \} : (Y_{j},\eta_{j}) \text{ is trivial } \} \quad\quad\quad\quad J_{c} = \{ 1, \ldots, k \} \setminus J_{d}
\]

Note that for $j \in J_{d}$ we have that $\stab_{*}\eta_{j}$ is a point mass $\delta_{Q}$ for some $Q \in S(G_{j})$.  As $\stab_{*}\eta_{j}$ is $\Gamma$-invariant this means $\gamma Q \gamma^{-1} = Q$ for all $\gamma$.  As $\Gamma$ is dense in $G_{j}$ this means $Q$ is normal in $G$ so $Q = G$ or $Q = \{ e \}$.  If $Q = \{ e \}$ then $\Gamma \actson (X,\nu)$ is essentially free since the projection map is faithful, in which case the proof is complete.  So we proceed with the premise that $\mathrm{proj}_{G_{j}}~\stab(x)$ is dense in $G_{j}$ almost everywhere for $j \in J_{d}$.

From here on, we assume that $J_{c} \ne \emptyset$ since if it is empty then the proof is complete as the projections of the stabilizers are dense to each $G_{j}$ almost everywhere.

Define the groups
\[
G_{c} = \prod_{j \in J_{c}} G_{j} \quad\quad\text{and}\quad\quad G_{d} = \prod_{j \in J_{d}} G_{j}
\]
so that $G = G_{c} \times G_{d}$ (after appropriate rearrangement).

For $j \in J_{c}$, Proposition \ref{P:cpt} says there exists a positive integer $m_{j}$ so that for the diagonal action $G_{j} \actson (Y_{j}^{m_{j}},\eta_{j}^{m_{j}})$ has the property that almost every stabilizer is compact or is discrete.  

Set $m = \max(m_{j} : j \in J_{c})$.  Then almost every stabilizer of $(Y_{j},\eta_{j})^{m}$ is discrete or compact.
Hence, for almost every $(y_{1},\ldots,y_{m}) \in Y_{j}^{m}$ we have that $\cap_{i=1}^{m} \stab(y_{i})$ is discrete or compact.  Then $\cap_{i=1}^{m} \overline{\mathrm{proj}_{G_{j}}~\stab_{\Gamma}(x_{i})}$ is discrete or compact for $\nu^{m}$-almost every $(x_{1},\ldots,x_{m}) \in X^{m}$.

So for $\nu^{m}$-almost every $(x_{1},\ldots,x_{m})$ we have $\overline{\mathrm{proj}_{G_{j}}~\cap_{i=1}^{m} \stab_{\Gamma}(x_{i})}$ is discrete or compact. 

\emph{Discrete Stabilizers}

The set $D = \{ z \in Y_{j}^{m} : \stab(z) \text{ is discrete} \}$ is $G_{j}$-invariant, and if it is positive measure we may consider $G_{j} \actson (D,(\eta^{m})^{*})$ where $(\eta^{m})^{*}$ is $\eta^{m}$ restricted to $D$ and normalized to be a probability measure.  Applying Proposition \ref{P:discrete}, we determine that $\stab(z) \cap \mathrm{proj}_{G_{j}}~\Gamma = \{ e \}$ for a.e.~$z \in D$.

Now if $\overline{\mathrm{proj}_{G_{j}}~\stab_{\Gamma}(x_{1},\ldots,x_{m})} \cap \mathrm{proj}_{G_{j}}~\Gamma = \{ e \}$ then $\mathrm{proj}_{G_{j}}~\stab_{\Gamma}(x_{1},\ldots,x_{m}) = \{ e \}$ and, as the projection is faithful on $\Gamma$, then $\stab_{\Gamma}(x_{1},\ldots,x_{m}) = \{ e \}$.

\emph{Compact Stabilizers}

For $z \notin D$, the corresponding $\mathrm{proj}_{G_{j}}~\stab_{\Gamma}(x_{1},\ldots,x_{m})$ is contained in a compact subgroup.  Therefore $\mathrm{proj}_{G_{j}}~\stab_{\Gamma}(x_{1},\ldots,x_{m})$ is contained in a compact subgroup of $G_{j}$ a.e.~for all $j \in J_{c}$.

\emph{The Diagonal Action $\Gamma \actson (X,\nu)^{m}$} 

Since $J_{c} \ne \emptyset$,  by Theorem \ref{T:projectinglatt} there is a $G_{d}$-space $(Z,\zeta)$ such that $\stab_{*}\zeta = s_{*}\nu$ where $s(x_{1},\ldots,x_{m}) = \overline{\mathrm{proj}_{G_{d}}~\stab_{\Gamma}(x_{1},\ldots,x_{m})}$.

Since $\stab_{\Gamma}(x_{1},\ldots,x_{m})$ is a discrete subgroup of $G_{c} \times G_{d}$ and is contained in $K \times G_{d}$ for some compact subgroup $K$ of $G_{c}$, Lemma \ref{L:topotrick} gives that $\mathrm{proj}_{G_{d}}~\stab_{\Gamma}(x_{1},\ldots,x_{m})$ is discrete in $G_{d}$.

\emph{The case when $J_{d} \ne \emptyset$}

Then $\stab_{*}\zeta$ is a stationary random subgroup of $G_{d}$ which is supported on subgroups of $\mathrm{proj}_{G_{d}}~\Gamma$.  But this means $\stab_{*}\zeta = \delta_{\{e\}}$ by Lemma \ref{L:srssupp} since the only normal subgroup of $G_{d}$ that is contained in $\Gamma$ is the trivial group (as $G_{d}$ is nontrivial when $J_{d}\ne\emptyset$).

So we have that $\Gamma \actson (X,\nu)^{m}$ is essentially free.  For any $\gamma \in \Gamma$ with $\nu(\mathrm{Fix}~\gamma) > 0$ also $\nu^{m}(\mathrm{Fix}_{X^{m}}~\gamma) > 0$.  So we have that $\Gamma \actson (X,\nu)$ is essentially free when $J_{d} \ne \emptyset$.

\emph{The case when $J_{d} = \emptyset$}

In this case, $\stab_{\Gamma}(x_{1},\ldots,x_{m})$ is contained in a compact subgroup of $G$ almost everywhere.

Then for any $\gamma \in \Gamma$ with $\nu(\mathrm{Fix}~\gamma) > 0$, since also $\nu^{m}(\mathrm{Fix}~\gamma) > 0$, we have that the group generated by $\gamma$ is discrete in $G$ but contained in a compact group hence is torsion.

If $\stab_{\Gamma}(x)$ contains a finitely generated infinite subgroup on a positive measure set then, as there only countably many finitely generated subgroups of $\Gamma$, there is some infinite group $A$ with $A \subseteq \stab_{\Gamma}(x)$ on a positive measure set.  Then $\nu(\mathrm{Fix}~A) > 0$ and so $(\eta_{1}\times\cdots\times\eta_{k})^{m}(\mathrm{Fix}~A) > 0$ but then $A$ must be contained in a compact group, which, since $A$ is discrete in $G$, contradicts that it is infinite.
Therefore when $J_{d} = \emptyset$, the $\Gamma$-stabilizers on $X$ are locally finite.
\end{proof}

\begin{proof}[Proof of Theorem \ref{T:latticeactions}]
Proposition \ref{P:nearly} tells us that either the stabilizers project densely to every proper subproduct or else the stabilizers are locally finite.
Selberg's Lemma states that $\Gamma$, being a lattice in a semisimple Lie group, is virtually torsion-free so there exists a finite index $\Gamma_{0} < \Gamma$ which is torsion-free.
Proposition \ref{P:locfin} then gives that the stabilizers are finite.
\end{proof}

\section{Stationary Actions of Lattices and Commensurators}

\subsection{Actions of Lattices in Connected Groups}

\begin{theorem}\label{T:latticewa}
Let $\Gamma < G = G_{1} \times \cdots \times G_{k}$ be a strongly irreducible lattice in a product of at least two connected simple Lie groups.

Let $\Gamma \actson (X,\nu)$ be an ergodic stationary action.

Then either $\stab_{\Gamma}(x)$ is finite almost everywhere or the action is weakly amenable and measure-preserving.
\end{theorem}
\begin{proof}
Theorem \ref{T:latticeactions} tells us that either the stabilizers are finite or they project densely to each simple factor.  When the projections are dense, Corollary \ref{C:latticedenseprojs} tells us the action is weakly amenable and measure-preserving.
\end{proof}

Weak amenability equates to orbit equivalence to a $\mathbb{Z}$-action (Connes-Feldman-Weiss \cite{CFW}).

\begin{theorem}\label{T:latticefi}
Let $\Gamma < G = G_{1} \times \cdots \times G_{k}$ be a strongly irreducible lattice in a product of connected simple Lie groups, at least two simple factors with one of higher-rank.

Let $\Gamma \actson (X,\nu)$ be an ergodic stationary action.

Then either $\stab_{\Gamma}(x)$ is finite almost everywhere or $\stab_{\Gamma}(x)$ is finite index almost everywhere.
\end{theorem}
\begin{proof}
Corollary 8.6 \cite{Cr17} tells us that if the action weakly amenable and measure-preserving then the stabilizers are finite index almost everywhere. 
\end{proof}

\subsection{Actions of Dense Commensurators}

\begin{theorem}\label{T:comm}
Let $\Gamma < G$ be a strongly irreducible lattice in a connected semisimple Lie groups, at least one of higher-rank.

Let $\Lambda < G$ be a dense commensurator of $\Gamma$ such that the relative profinite completion $\rpf{\Lambda}{\Gamma}$ is a product of simple uncountable totally disconnected locally compact groups with the Howe-Moore property.

Let $\Lambda \actson (X,\nu)$ be an ergodic stationary action.

Then either $\stab_{\Lambda}(x)$ is finite almost everywhere or $\stab_{\Lambda}(x)$ is finite index almost everywhere.
\end{theorem}
\begin{proof}
By Proposition \ref{P:relpro}, the restriction of the action to $\Gamma$ is stationary.  By Theorem \ref{T:latticefi}, the restriction of the action to $\Gamma$ has stabilizers which are finite or finite index almost everywhere (on each $\Gamma$-ergodic component one or the other must occur).  For $\lambda \in \Lambda$, we know $\Gamma \cap \lambda\Gamma\lambda^{-1}$ is finite index in both $\Gamma$ and $\lambda\Gamma\lambda^{-1}$.  If $[\Gamma : \stab_{\Gamma}(x)] < \infty$ then $[\Gamma : \stab(x) \cap \Gamma] < \infty$ and so $[\lambda\Gamma\lambda^{-1} : \lambda\stab(x)\lambda^{-1} \cap \lambda\Gamma\lambda^{-1}] < \infty$.  Using that $\Gamma \cap \lambda\Gamma\lambda^{-1}$ is finite index then $[\Gamma \cap \lambda\Gamma\lambda^{-1} : \lambda\stab(x)\lambda^{-1} \cap \lambda\Gamma\lambda^{-1} \cap \Gamma] < \infty$ and so $[\Gamma : \lambda\stab(x)\lambda^{-1}\cap\Gamma] < \infty$.  But this just says $[\Gamma : \stab_{\Gamma}(\lambda x)] < \infty$.  So the set where the $\Gamma$-stabilizers are finite index is $\Lambda$-invariant hence by $\Lambda$-ergodicity the $\Gamma$-stabilizers are either finite almost everywhere or finite index almost everywhere.

\emph{Finite $\Gamma$-Stabilizers}

Let $H = \rpf{\Lambda}{\Gamma}$ be the relative profinite completion.  Then $\Lambda < G \times H$ is a strongly irreducible irreducible lattice (Theorem \ref{T:cl}).

When $\stab_{\Gamma}(x)$ is finite, since $\mathrm{proj}_{H}~\stab_{\Lambda}(x) \cap K = \mathrm{proj}_{H}~\stab_{\Lambda}(x) \cap \mathrm{proj}_{H}~\Gamma$ by Proposition 6.1.3 \cite{CP17},
\[
\mathrm{proj}_{H}~\stab_{\Lambda}(x) \cap K = \mathrm{proj}_{H}~(\stab_{\Lambda}(x) \cap \Gamma) = \mathrm{proj}_{H}~\stab_{\Gamma}(x)
\]
is finite.  So $\mathrm{proj}_{H}~\stab_{\Lambda}(x)$ is discrete in $H$ as $K$ is open.  In particular, $\overline{\mathrm{proj}_{H}~\stab_{\Lambda}(x)} = \mathrm{proj}_{H}~\stab_{\Lambda}(x) \subseteq \mathrm{proj}_{H}~\Lambda$ almost everywhere.

Treating $\Lambda$ as an irreducible lattice in $G \times H$, Theorem \ref{T:projectinglatt} tells us there is an ergodic $H$-space $(Y,\eta)$ such that $\stab_{*}\eta = s_{*}\nu$ where $s : X \to S(H)$ is the $\Lambda$-map $s(x) = \overline{\mathrm{proj}_{H}~\stab_{\Lambda}(x)}$.

Then $\stab_{*}\eta$ is a random subgroup of $H$ which is supported on subgroups of $\mathrm{proj}_{H}~\Lambda$.  Lemma \ref{L:srssupp} then says it is supported on subgroups of $N$ for some $N \normal H$ with $N \subseteq \mathrm{proj}_{H}~\Lambda$.

Since $H = \prod H_{j}$ is a product of simple groups, $N$ is a proper subproduct.  But $N \subseteq \mathrm{proj}_{H}~\Lambda$ so if $N$ is nontrivial then there some simple $H_{j}$ which is contained in $\mathrm{proj}_{H}~\Lambda$ meaning that $H_{j}$ is countable, contradicting our hypothesis.

So $N$ is trivial meaning that $\overline{\mathrm{proj}_{H}~\stab_{\Lambda}(x)} = \{ e \}$ almost everywhere.  The kernel of $\mathrm{proj}_{H}$ is contained in $\Gamma$ hence is finite (by Margulis' Normal Subgroup Theorem if it is infinite then it is finite index; but that would make $H$ discrete hence countable).  Therefore $\stab_{\Lambda}(x)$ is finite almost everywhere.

\emph{Finite Index $\Gamma$-Stabilizers}

When $\stab_{\Gamma}(x)$ is finite index in $\Gamma$, we have that $\mathrm{proj}_{H}~\stab_{\Lambda}(x)$ contains a finite index subgroup of $\mathrm{proj}_{H}~\Gamma$ and therefore $\overline{\mathrm{proj}_{H}~\stab_{\Lambda}(x)}$ contains a finite index subgroup of $\overline{\mathrm{proj}_{H}~\Gamma} = K$.

As $K$ is a compact open subgroup, then $\overline{\mathrm{proj}_{H}~\stab_{\Lambda}(x)}$ contains a compact open subgroup and hence is open almost everywhere.

Write $H = \prod H_{j}$ and then $\overline{\mathrm{proj}_{H_{j}}~\stab_{\Lambda}(x)}$ is an open subgroup of $H_{j}$ almost everywhere for each $j$ (as the projection maps are open maps).  Since each $H_{j}$ has the Howe-Moore property, the only proper open subgroups are compact.
Write
\[
J_{c} = \{ j : \overline{\mathrm{proj}_{H_{j}}~\stab(x)} \text{ is compact a.e.} \} \quad J_{d} = \{ j : j \notin J_{c} \} = \{ j : \overline{\mathrm{proj}_{H_{j}}~\stab(x)} = H_{j}~\text{a.e.} \}
\]
and define the groups
\[
H_{c} = \prod_{j \in J_{c}} H_{j} \quad\quad\quad\quad H_{d} = \prod_{j \in J_{d}} H_{j}
\]

If $J_{c} \ne \emptyset$ then  $\stab_{\Lambda}(x) < G \times H_{d} \times H_{c}$ is discrete and $\stab_{\Lambda}(x) < G \times H_{d} \times K_{x}$ for some compact $K_{x}$ so Lemma \ref{L:topotrick} says that $\mathrm{proj}_{G\times H_{d}}~\stab_{\Lambda}(x)$ is discrete in $G\times H_{d}$.

Since $\Lambda$ is strongly irreducible as a lattice in $G \times H$, and since $H_{d}$ is a proper subproduct of $H$, $\Lambda$ is a dense commensurator of a lattice in $G \times H_{d}$ (Proposition \ref{P:relpro}) so Theorem \ref{T:projectingcomm} tells us there is an ergodic $G\times H_{d}$-space $(Y,\eta)$ such that $\stab_{*}\eta = s_{*}\nu$ where $s(x) = \overline{\mathrm{proj}_{G\times H_{d}}~\stab_{\Lambda}(x)}$ is the $\Lambda$-map $s : X \to S(G\times H_{d})$.

Then $\stab_{*}\eta$ is a random subgroup of $G\times H_{d}$ supported on subgroups of $\mathrm{proj}_{G\times H_{d}}~\Lambda$.  Lemma \ref{L:srssupp} then says it is supported on subgroups of $N$ for some $N \normal G$ with $N \subseteq \mathrm{proj}_{G}~\Lambda$.  But no proper subproduct of $G \times H_{d}$ can be contained in a countable group so $N$ is trivial.  This means that $\mathrm{proj}_{G\times H_{d}}~\stab_{\Lambda}(x) = \{ e \}$ contradicting that the $\Gamma$-stabilizers are finite index.

So $J_{c} = \emptyset$ and $H_{d} = H$.  Then $\stab_{\Lambda}(x)$ projects densely into each $H_{j}$ almost everywhere.  Since it also projects densely into each simple factor of $G$ (as the $\Gamma$-stabilizers do), Corollary \ref{C:latticedenseprojs} then tells us that $\Lambda \actson (X,\nu)$ is weakly amenable and measure-preserving, so by Corollary 8.6 \cite{Cr17} (every simple $H_{j}$ acts ergodically on the induced space as the projections of the stabilizers are dense), it is essentially transitive meaning that $\stab_{\Lambda}(x)$ is finite index in $\Lambda$ a.e.
\end{proof}

\begin{corollary}\label{C:commensurators}
Let $G$ be a semisimple connected Lie group with finite center, with at least one simple factor of higher-rank, and $\Gamma < G$ be a finitely generated, strongly irreducible lattice.

Let $\Lambda$ be a dense commensurator of $\Gamma$.
Then every stationary action of $\Lambda$ has finite stabilizers almost everywhere or has finite index stabilizers almost everywhere.
\end{corollary}
\begin{proof}
By Margulis' Arithmeticity Theorem \cite{Margulis}, $\Gamma$ is the $\mathbb{Z}$-points of some semisimple algebraic group $\mathbf{G}$ (up to finite index and finite kernels).  The commensurator of $\mathbf{G}(\mathbb{Z})$ is $\mathbf{G}(\mathbb{Q})$.  The relative profinite completion $\rpf{\Lambda}{\Gamma}$ is therefore of the form $\prod \mathbf{G}(\mathbb{Q}_{p})$ by Theorem 10.2 in \cite{CP17}.  Since simple $p$-adic groups are Howe-Moore, Theorem \ref{T:comm} gives the conclusion.
\end{proof}

\subsection{Actions of Lattices in General Semisimple Groups}

\begin{corollary}\label{C:generallattices}
Let $\Gamma < G$ be a strongly irreducible lattice in a semisimple group with finite center and no compact factors, with at least one connected higher-rank simple factor.

Every stationary action of $\Gamma$ either has finite stabilizers or finite index stabilizers.
\end{corollary}
\begin{proof}
Write $G = G_{0} \times H$ where $G_{0}$ is connected and $H$ is totally disconnected.  If $H$ is trivial then Theorem \ref{T:latticefi} gives the conclusion; if not then by Proposition \ref{P:relpro} the projection of $\Gamma$ to $G$ is a dense commensurator of a strongly irreducible lattice $\Gamma_{0} < G_{0}$ and any stationary action of $\Gamma$ is stationary for $\Gamma_{0}$ so Corollary \ref{C:commensurators} then gives the result (the case when $G$ is a simple connected higher-rank group is covered by \cite{BH19}).
\end{proof}

\ifthenelse{\equal{\includeappendix}{true}}{%
\appendix
\section{The Stationary Intermediate Factor Theorem for Dense Commensurators}

For completeness, we include a stationary form of the factor theorem for commensurators, though as we did not need it above we include it as an appendix.

Such a theorem will be needed to handle e.g.~actions of tree lattices, though we do not elaborate nor attempt that here.

\begin{theorem}\label{T:uniquemap}
Let $G$ be a locally compact second countable group and $\mu$ an admissible measure on $G$.  Let $(B,\beta)$ be the Poisson boundary for $(G,\mu)$ and $(X,\nu)$ be an ergodic $(G,\mu)$-space.  Let $(Z,\zeta)$ be a $G$-space.

Let $\pi : (B,\beta) \curlyvee (X,\nu) \to (Z,\zeta)$ and $\varphi : (Z,\zeta) \to (X,\nu)$ and $\pi^{\prime} : (B,\beta) \curlyvee (X,\nu) \to (Z,\zeta^{\prime})$ and $\varphi^{\prime} : (Z,\zeta^{\prime}) \to (X,\nu)$ such that $\varphi \circ \pi = \varphi^{\prime} \circ \pi^{\prime}$ is the projection map.

Then $\zeta = \zeta^{\prime}$ and $\pi = \pi^{\prime}$ almost everywhere (and $\varphi = \varphi^{\prime}$ almost everywhere).
\end{theorem}
\begin{proof}
Since $G \actson (B,\beta) \curlyvee (X,\nu)$ is ergodic hence so is $G \actson (Z,\zeta)$.  As the same holds for $\zeta^{\prime}$ then $\zeta = \zeta^{\prime}$ since ergodic stationary measures are extremal in the (convex) set of stationary measures (Corollary 2.7 \cite{BS}).

For $\omega \in G^{\mathbb{N}}$, consider the maps $\pi : (B \times X, \beta_{\omega} \times \nu_{\omega}) \to (Z,\zeta_{\omega})$ and $\varphi : (Z,\zeta_{\omega}) \to (X,\nu_{\omega})$ (that these maps exist for $\mu^{\mathbb{N}}$-almost every $\omega$ is \cite{FG}).

Let $(\zeta_{\omega})_{x}$ be the disintegration of $\zeta_{\omega}$ over $\nu_{\omega}$.  Since $\beta_{\omega} = \delta_{b(\omega)}$ where $b : G^{\mathbb{N}} \to B$ is the boundary map, we then have $\pi : \delta_{b} \times \delta_{x} \to (\zeta_{\omega})_{x}$ so $(\zeta_{\omega})_{x} = \delta_{\pi(b,x)}$ almost everywhere.

Thus, for $\mu^{\mathbb{N}}$-almost every $\omega$ and $\nu_{\omega}$-almost every $x$ we have $(\zeta_{\omega})_{x} = \delta_{\pi(b(\omega),x)}$.  The same would hold for $\pi^{\prime}$ though, so we conclude that $\pi(b(\omega),x) = \pi^{\prime}(b(\omega),x)$ for $\mu^{\mathbb{N}}$-almost every $\omega$ and $\nu_{\omega}$-almost every $x$.  But this simply means that $\pi(b,x) = \pi^{\prime}(b,x)$ for $\beta \curlyvee \nu$-almost every $(b,x)$.
\end{proof}

\begin{theorem}\label{T:iftdense}
Let $\Gamma < G$ be a lattice in a locally compact second countable group and $\Lambda < G$ a dense commensurator.

Let $\Lambda \actson (X,\nu)$ be an ergodic $\Lambda$-space such that the restriction of the action to $\Gamma$ is stationary.

Let $(B,\beta)$ be the Poisson boundary of $G$.  Let $(Z,\zeta)$ be a $(\Gamma,\mu)$-space with $\Gamma$-maps
\[
\pi : (B,\beta) \curlyvee (X,\nu) \to (Z,\zeta) \quad\quad\quad\quad \varphi: (Z,\zeta) \to (X,\nu)
\]
that compose to the projection map.

Assume that $(Z,\zeta)$ is orbital over $(X,\nu)$: for almost every $z$ and all $\gamma \in \Gamma$, if $\gamma \varphi(z) = \varphi(z)$ then $\gamma z = z$.

For each $\lambda \in \Lambda$, we have $\pi(\lambda(b,x)) = \pi(b,x)$ for $\beta\curlyvee\nu$-almost every $(b,x)$ such that $\lambda x = x$.
\end{theorem}
\begin{proof}
We remark that what follows is near identical to the proof of Theorem 4.19 in \cite{CP17} but as the results there are stated for measure-preserving $(X,\nu)$, we cannot cite them directly.

Fix $\lambda \in \Lambda$.  Let $E = \{ x \in X : \lambda x \in \Gamma x \}$.  Define the map $\theta_{\lambda} : \varphi^{-1}(E) \to Z$ by choosing $\gamma \in \Gamma$ so that $\lambda \varphi(z) = \gamma \varphi(z)$ then define $\theta_{\lambda}(z) = \gamma z$ (which is well-defined since $Z$ is orbital).

\emph{Claim}: $\theta_{\lambda}(\varphi^{-1}(E)) = \varphi^{-1}(\lambda E)$

For $z \in \theta_{\lambda}(\varphi^{-1}(E))$ write $z = \theta_{\lambda}(w)$ for $w \in \varphi^{-1}(E)$ so that $\lambda \varphi(w) = \gamma \varphi(w)$ for some $\gamma$.  Then $z = \gamma w$ by definition of $\theta_{\lambda}$ so $\varphi(z) = \varphi(\gamma w) = \gamma \varphi(w) = \lambda \varphi(w) \in \lambda E$.

For $z \in \varphi^{-1}(\lambda E)$ write $\varphi(z) = \lambda x$ for $x \in E$ and then there exists $\gamma \in \Gamma$ with $\varphi(z) = \lambda x = \gamma x$.  As $\varphi(\gamma^{-1}z) = \gamma^{-1}\varphi(z) = x \in E$ and $\lambda x = \gamma x$ then $\theta_{\lambda}(\gamma^{-1}z) = \gamma (\gamma^{-1}z) = z$.  So $z = \theta_{\lambda}(\gamma^{-1}z) \in \theta_{\lambda}(\varphi^{-1}(E))$.

\emph{Claim}:  Set $\Gamma_{0} = \Gamma \cap \lambda^{-1}\Gamma\lambda$ which is a lattice in $G$.  The set $E$ is $\Gamma_{0}$-invariant.

That $\Gamma_{0}$ is a lattice is immediate from the fact that $\Lambda$ commensurates $\Gamma$ meaning $\Gamma_{0}$ is finite index in $\Gamma$.  For $\gamma_{0} \in \Gamma_{0}$ and $x \in E$, writing $\gamma_{0} = \lambda^{-1}\gamma\lambda$ for $\gamma \in \Gamma$ we have $\lambda \gamma_{0}x = \lambda\lambda^{-1}\gamma\lambda x = \gamma \lambda x \in \gamma \Gamma x = \Gamma x$ meaning that if $\lambda x \in \Gamma x$ then $\lambda \gamma_{0} x \in \Gamma x$ and so $E$ is $\Gamma_{0}$-invariant.

\emph{Claim}: There exists $\theta_{\lambda}^{-1} : \theta_{\lambda}(\varphi^{-1}(E)) \to E$ such that $\theta_{\lambda}^{-1}\theta_{\lambda} = \mathrm{id}$ on $\varphi^{-1}(E)$ and $\theta_{\lambda}\theta_{\lambda}^{-1} = \mathrm{id}$ on $\theta_{\lambda}(\varphi^{-1}(E))$

Let $w \in \theta_{\lambda}(\varphi^{-1}(E))$.  Then $w = \theta_{\lambda}(z)$ for some $z \in \varphi^{-1}(E)$ so $w = \gamma z$ for some $\gamma \in \Gamma$ such that $\lambda \varphi(z) = \gamma \varphi(z)$.  Note that if $\gamma,\gamma^{\prime} \in \Gamma$ are both such that $\lambda \varphi(z) = \gamma \varphi(z) = \gamma^{\prime} \varphi(z)$ then $\gamma^{-1}\gamma^{\prime} \varphi(z) = \varphi(z)$ so as $Z$ is orbital then $\gamma^{-1}\gamma^{\prime} z = z$.  Define $\theta_{\lambda}^{-1}(w) = \gamma^{-1}w$.  This is then well-defined since $(\gamma^{\prime})^{-1}w = (\gamma^{\prime})^{-1}\gamma \gamma^{-1}w = (\gamma^{\prime})^{-1}\gamma z = z = \gamma^{-1}w$ because $\gamma^{-1}\gamma^{\prime} z = z$.  Then $\theta_{\lambda}^{-1}(\theta_{\lambda}(z)) = \theta_{\lambda}^{-1}(w) = z$ and $\theta_{\lambda}(\theta_{\lambda}^{-1}(w)) = \theta_{\lambda}(z) = w$ hence the proof is complete (since $\theta_{\lambda}$ maps onto its image).

\emph{The new maps}: Define the map $\pi_{\lambda} : B \times X \to Z$ as follows: for $(b,x) \in p^{-1}(E)$ set $\pi_{\lambda}(b,x) = \theta_{\lambda}^{-1}(\pi(\lambda (b,x)))$ and for $y \notin p^{-1}(E)$ set $\pi_{\lambda}(b,x) = \pi(b,x)$.  Likewise define the map $\varphi_{\lambda} : Z \to X$ by $\varphi_{\lambda}(z) = \lambda^{-1}\varphi(\theta_{\lambda} (z))$ for $z$ such that $\varphi(z) \in E$ and $\varphi_{\lambda}(z) = \varphi(z)$ for $z$ such that $\varphi(z) \notin E$.

\emph{Claim}:  $\varphi_{\lambda} \circ \pi_{\lambda} = \varphi \circ \pi = p$ is the projection map and both $\pi_{\lambda}$ and $\varphi_{\lambda}$ are $\Gamma_{0}$-equivariant.

Note that in fact $\varphi_{\lambda} = \varphi$ since for $z \in \varphi^{-1}(E)$ and $\gamma \in \Gamma$ such that $\lambda\varphi(z) = \gamma\varphi(z)$ we have that $\lambda^{-1} \varphi(\gamma z) = \lambda^{-1}\gamma \varphi(z) = \varphi(z)$ but we will find it helpful to distinguish these maps since the measures $\pi_{*}\eta$ and $(\pi_{\lambda})_{*}\eta$ may be distinct and we will be treating $\varphi_{\lambda}$ as a map $(Z,(\pi_{\lambda})_{*}\eta) \to (X,\nu)$ and $\varphi$ as a map $(Z,\pi_{*}\eta) \to (X,\nu)$.

Now for $(b,x)$ such that $p(b,x) \in E$, observe that
\[
\varphi_{\lambda}(\pi_{\lambda}(b,x)) = \lambda^{-1}\varphi(\theta_{\lambda} \theta_{\lambda}^{-1} \pi(\lambda (b,x)))
= \lambda^{-1} \varphi(\pi(\lambda (b,x))) = \lambda^{-1} p(\lambda (b,x)) = p(b,x)
\]
since $p$ is $\Lambda$-equivariant.  Clearly for $(b,x)$ such that $p(b,x) \notin E$ we have $\varphi_{\lambda}(\pi_{\lambda}(b,x)) = \varphi_{\lambda}(\pi(b,x)) = \varphi(\pi(b,x)) = p(bmx)$.  Hence $\varphi_{\lambda} \circ \pi_{\lambda} = p$.

Observe that for $\gamma_{0} \in \Gamma_{0}$, writing $\gamma_{0} = \lambda^{-1}\gamma\lambda$ for some $\gamma \in \Gamma$, we have that for $(b,x)$ such that $p(b,x) \in E$, also $p(\gamma_{0}(b,x)) = \gamma_{0} p(b,x) \in E$ since $E$ is $\Gamma_{0}$-invariant, so
\begin{align*}
\pi_{\lambda}(\gamma_{0}(b,x)) &= \theta_{\lambda}^{-1} \pi(\lambda \gamma_{0}(b,x))
= \theta_{\lambda}^{-1} \pi(\gamma \lambda (b,x))
= \theta_{\lambda}^{-1} \gamma \pi(\lambda (b,x)) \\
&= \theta_{\lambda}^{-1} \gamma \theta_{\lambda} \theta_{\lambda}^{-1} \pi(\lambda (b,x))
= \theta_{\lambda}^{-1} \gamma \theta_{\lambda} \pi_{\lambda}(b,x).
\end{align*}
Now observe that for $z$ such that $\varphi(z) \in E$ (which includes $\pi_{\lambda}(b,x)$ for $p(b,x) \in E$), write $\gamma^{\prime} \in \Gamma$ such that $\theta_{\lambda} z = \gamma^{\prime}z$ and observe that then $\lambda \varphi(z) = \gamma^{\prime} \varphi(z)$ and so
\[
\gamma \gamma^{\prime} \gamma_{0}^{-1} \varphi(\gamma_{0} z)
= \gamma \gamma^{\prime} \varphi(z) = \gamma \lambda \varphi(z) = \lambda \gamma_{0} \varphi(z)
= \lambda \varphi(\gamma_{0} z)
\]
which in turn means that
\[
\theta_{\lambda} (\gamma_{0} z) = \gamma \gamma^{\prime} \gamma_{0}^{-1} (\gamma_{0} z) = \gamma \gamma^{\prime} z = \gamma \theta_{\lambda} (z)
\]
and therefore
\[
\pi_{\lambda}(\gamma_{0}(b,x)) =  \theta_{\lambda}^{-1} \gamma \theta_{\lambda} \pi_{\lambda}(b,x) 
= \theta_{\lambda}^{-1}\theta_{\lambda}\gamma_{0}\pi_{\lambda}(b,x) = \gamma_{0} \pi_{\lambda}(b,x).
\]
Of course, for $(b,x)$ such that $p(b,x) \notin E$ we have that $\gamma_{0} (b,x) \notin E$ and so
\[
\pi_{\lambda}(\gamma_{0} (b,x)) = \pi(\gamma_{0} (b,x)) = \gamma_{0} \pi(b,x) = \gamma_{0} \pi_{\lambda}(b,x)
\]
and we conclude that $\pi_{\lambda}$ is $\Gamma_{0}$-equivariant.  Note that $\varphi_{\lambda} = \varphi$ so $\varphi_{\lambda}$ is likewise $\Gamma_{0}$-equivariant.

\emph{Completing the proof}:  The claims together give us $\Gamma_{0}$-maps of $\Gamma_{0}$-spaces $\pi : (B,\beta) \curlyvee (X,\nu) \to (Z,\zeta)$ and $\varphi : (Z,\zeta) \to (X,\nu)$ and $\pi_{\lambda} : (B,\beta) \curlyvee (X,\nu) \to (Z,\zeta_{\lambda})$ and $\varphi_{\lambda} : (Z,\zeta_{\lambda}) \to (X,\nu)$ where $\zeta_{\lambda} = \pi_{*}(\beta \curlyvee \nu)$ which is in the same measure class as $\zeta$.

Theorem \ref{T:uniquemap} then says that $\zeta_{\lambda} = \zeta$ and $\pi = \pi_{\lambda}$ almost everywhere.  For almost every $(b,x)$, we then have $\pi(\lambda(b,x)) = \theta_{\lambda}\theta_{\lambda^{-1}}\pi(\lambda (b,x)) = \theta_{\lambda}\pi_{\lambda}(b,x) = \theta_{\lambda}\pi(b,x)$.

In particular, when $\lambda x = x$ meaning $\theta_{\lambda}\pi(b,x) = \pi(b,x)$ (as $Z$ is orbital), we then have $\pi(\lambda(b,x)) = \pi(b,x)$.
\end{proof}}{\relax}

\vspace{-2em}
\dbibliography{references}


\providecommand{\bysame}{\leavevmode\hbox to3em{\hrulefill}\thinspace}
\providecommand{\MR}{\relax\ifhmode\unskip\space\fi MR }
\providecommand{\MRhref}[2]{%
  \href{http://www.ams.org/mathscinet-getitem?mr=#1}{#2}
}
\providecommand{\href}[2]{#2}
\begin{thebibliography}{BRHP20}

\bibitem[BH19]{BH19}
R.~Boutonnet and C.~Houdayer, \emph{Stationary characters on lattices of
  semisimple lie groups}, Preprint, 2019.

\bibitem[BRHP20]{BBHP20}
U.~Bader, Boutonnet R., C~Houdayer, and J.~Peterson, \emph{Charmenability of
  arithmetic groups of product type}, Preprint, 2020.

\bibitem[BS06]{BS}
Uri Bader and Yehuda Shalom, \emph{Factor and normal subgroup theorems for
  lattices in products of groups}, Inventiones Mathematicae \textbf{163}
  (2006), no.~2, 415--454.

\bibitem[CFW81]{CFW}
A.~Connes, J.~Feldman, and B.~Weiss, \emph{An amenable equivalence relation is
  generated by a single transformation}, Ergodic Theory and Dynamical Systems
  \textbf{1} (1981), no.~4, 431--450.

\bibitem[CP17]{CP17}
Darren Creutz and Jesse Peterson, \emph{Stabilizers of ergodic actions of
  lattices and commensurators}, Transactions of the American Mathematical
  Society (2017), no.~369, 4119--4166.

\bibitem[Cre17]{Cr17}
Darren Creutz, \emph{Stabilizers of actions of lattices in products of groups},
  Journal of Ergodic Theory and Dynamical Systems \textbf{37} (2017),
  1133--1186.

\bibitem[Cre18]{vanderbilttalk}
\bysame, \emph{Essential freeness of stationary actions of lattices},
  Vanderbilt Subfactor Seminar, 2018.

\bibitem[CS14]{CS14}
Darren Creutz and Yehuda Shalom, \emph{A normal subgroup theorem for
  commensurators of lattices}, Groups, Geometry and Dynamics \textbf{8} (2014),
  1--22.

\bibitem[FG10]{FG}
Hillel Furstenberg and Eli Glasner, \emph{Stationary dynamical systems},
  Dynamical Numbers---Interplay Between Dynamical Systems and Number Theory,
  Contemporary Mathematics, vol. 532, American Mathematical Society, 2010,
  pp.~1--28.

\bibitem[Fur00]{furman}
Alex Furman, \emph{Random walks on groups and random transformations}, Handbook
  of Dynamical Systems, Elsevier, 2000.

\bibitem[Gla03]{glasner}
Eli Glasner, \emph{Ergodic theory via joinings}, Mathematical Surveys and
  Monographs, vol. 101, American Mathematical Society, 2003.

\bibitem[GW14]{GW14}
Eli Glasner and Benjamin Weiss, \emph{Uniformly recurrent subgroups}, To
  appear, 2014.

\bibitem[HK64]{HK}
P.~Hall and C.~R. Kulatilaka, \emph{A property of locally finite groups},
  Journal of the London Mathematical Society \textbf{39} (1964), 235--239.

\bibitem[HM79]{HM79}
R.E. Howe and C.C. Moore, \emph{Asymptotic properties of unitary
  representations}, Journal of Functional Analysis \textbf{32} (1979), 72--96.

\bibitem[Kak51]{kakutani}
Shizuo Kakutani, \emph{Random ergodic theorems and {M}arkov processes with a
  stable distribution}, Proceedings of the Second Berkeley Symposium on
  Mathematics, Statistics and Probability, University of California Press,
  1951, pp.~247--261.

\bibitem[Mar79]{Margulis}
Gregory Margulis, \emph{Finiteness of quotient groups of discrete subgroups},
  Funktsional\'nyi Analiz i Ego Prilozheniya \textbf{13} (1979), 28--39.

\bibitem[Mar91]{Ma91}
\bysame, \emph{Discrete subgroups of semisimple {L}ie groups}, Springer-Verlag,
  1991.

\bibitem[NZ99a]{nevozimmerpreprint}
Amos Nevo and Robert Zimmer, \emph{A generalization of the intermediate factor
  theorem}, Preprint, 1999.

\bibitem[NZ99b]{NZ}
\bysame, \emph{Homogenous projective factors for actions of semi-simple lie
  groups}, Inventiones Mathematicae (1999), 229--252.

\bibitem[Pet14]{peterson}
J.~Peterson, \emph{Character rigidity for lattices in higher-rank groups},
  Preprint, 2014.

\bibitem[SZ94]{SZ}
Garrett Stuck and Robert Zimmer, \emph{Stabilizers for ergodic actions of
  higher rank semisimple groups}, The Annals of Mathematics \textbf{139}
  (1994), no.~3, 723--747.

\bibitem[Ver11]{Vershik}
A.~M. Vershik, \emph{Nonfree actions of countable groups and their characters},
  Journal of Mathematical Science \textbf{174} (2011), no.~1, 1--6.

\bibitem[Zim77]{Zi77}
Robert Zimmer, \emph{Hyperfinite factors and amenable group actions},
  Inventiones Mathematicae \textbf{41} (1977), no.~1, 23--31.

\bibitem[Zim84]{Zim84}
\bysame, \emph{Ergodic theory and semisimple groups}, Birkhauser, 1984.

\end{thebibliography}
\end{document}